\documentclass[12pt]{amsart}

\usepackage{enumerate, amsmath, amsthm, amsfonts, amssymb, xy,  mathrsfs, graphicx, paralist, fancyvrb, eucal}
\usepackage[usenames, dvipsnames]{xcolor}
\usepackage[margin=1in]{geometry} 
\usepackage[bookmarks, colorlinks=true, linkcolor=blue, citecolor=blue, urlcolor=blue]{hyperref}

\input xy
\xyoption{all}

\numberwithin{equation}{section}
\newtheorem{theorem}[equation]{Theorem}

\newtheorem{proposition}[equation]{Proposition}
\newtheorem{lemma}[equation]{Lemma}
\newtheorem{corollary}[equation]{Corollary}

\theoremstyle{definition}
\newtheorem{rmk}[equation]{Remark}
\newenvironment{remark}[1][]{\begin{rmk}[#1] \pushQED{\qed}}{\popQED \end{rmk}}
\newtheorem{example}[equation]{Example}
\newtheorem{defn}[equation]{Definition}
\newenvironment{definition}[1][]{\begin{defn}[#1]\pushQED{\qed}}{\popQED \end{defn}}

\newenvironment{subeqns}[1][]{\addtocounter{equation}{-1}
\begin{subequations}

}{\end{subequations}}

\newcommand{\cA}{\mathcal{A}}

\newcommand{\cB}{\mathcal{B}}

\newcommand{\bF}{\mathbf{F}}

\newcommand{\rH}{\mathrm{H}}

\newcommand{\cI}{\mathcal{I}}

\newcommand{\cL}{\mathcal{L}}

\newcommand{\bP}{\mathbf{P}}

\newcommand{\bS}{\mathbf{S}}

\newcommand{\cV}{\mathcal{V}}

\newcommand{\sV}{\mathscr{V}}

\newcommand{\bZ}{\mathbf{Z}}

\newcommand{\bd}{\mathbf{d}}

\newcommand{\be}{\mathbf{e}}

\newcommand{\bk}{\mathbf{k}}

\renewcommand{\phi}{\varphi}
\renewcommand{\emptyset}{\varnothing}

\newcommand{\arxiv}[1]{\href{http://arxiv.org/abs/#1}{{\tt arXiv:#1}}}

\renewcommand{\hom}{\operatorname{Hom}}

\DeclareMathOperator{\Sym}{Sym}

\DeclareMathOperator{\Tor}{Tor}

\newcommand{\SL}{\mathbf{SL}}

\newcommand{\Sec}{\mathrm{Sec}}
\newcommand{\multi}[2]{(\bZ_{\ge 0}^{#1})_{#2}}
\newcommand{\df}[1]{{\bf \textsf{#1}}}

\title[Syzygies of bounded rank symmetric tensors]{Syzygies of bounded rank symmetric tensors \\are generated in bounded degree}
\date{November 21, 2016}

\author{Steven V Sam}
\address{Department of Mathematics, University of Wisconsin, Madison, WI}
\email{\href{mailto:svs@math.wisc.edu}{svs@math.wisc.edu}}
\urladdr{\url{http://math.wisc.edu/~svs/}}
\thanks{SS was partially supported by NSF DMS-1500069 and Iuventus  Plus  grant
0301/IP3/2015/73 of the Polish Ministry of Science.}


\begin{document}

\maketitle

\begin{abstract}
We study the syzygies of secant ideals of Veronese subrings of a fixed commutative graded algebra over a field of characteristic $0$. One corollary is that the degrees of the minimal generators of the $i$th syzygy module of the coordinate ring of the $r$th secant variety of any Veronese embedding of a projective scheme $X$ can be bounded by a constant that only depends on $i$, $r$, and $X$, and not on the choice of the Veronese embedding.
\end{abstract}

\section{Introduction}

In \cite{secver}, we studied the ideal of the $r$th secant variety of the $d$th Veronese embedding of a projective space (over a field of characteristic $0$), and proved the existence of a bound on the degrees of its minimal generators which is independent of $d$ (but depends on $r$, as it must). In fact, the ideal of a secant variety has a purely algebraic definition, and so makes sense for any $\bZ_{\ge 0}$-graded commutative ring $B$ (the projective space case is when $B$ is a polynomial ring) and we also proved analogous results in this level of generality. The basic idea in \cite{secver} is to recast the existence of these bounds to showing that a particular ideal in a new algebraic structure $\cB$ is finitely generated. Ultimately, one shows that all ideals in $\cB$ are finitely generated, i.e., $\cB$ is noetherian.

A natural question is if all finitely generated modules over $\cB$ are noetherian; however, from the perspective of \cite{secver} it is not clear how to define modules. Our first goal in the current paper is to translate  \cite{secver} into the language of functor categories in the spirit of \cite{fi-modules, catgb}. This has the downside of being more abstract, but can handle arbitrary modules with ease, rather than just ideals. An immediate application of our expanded toolkit is that we can study free resolutions of secant ideals (for some reasons of interest, see \cite{oeding}). In particular, we show that the $i$th syzygy module of the coordinate ring of the $r$th secant variety of the $d$th Veronese embedding of projective space (whose space of minimal generators is the $i$th Tor group with the residue field) is generated in bounded degree where the bound is independent of $d$ (see Corollary~\ref{cor:sec-syz}). The case $i=1$ corresponds to the results of \cite{secver} mentioned above. Furthermore, our techniques show that there are operators acting on these Tor groups in a way which makes them finitely generated (see Remark~\ref{rmk:delta-mod}) when considered all at once.

\subsection*{Notation}

Given a nonnegative integer $n$, let $[n] = \{1, \dots, n\}$ (by convention, $[0] = \emptyset$). The group of permutations of $[n]$ is denoted $\Sigma_n$.  Throughout, $\bk$ will denote a commutative ring. In some places we will restrict it to be a field of characteristic $0$. Given a set $X$, $\bk[X]$ is the free $\bk$-module with basis $X$. On the other hand, $\bk[x_1,\dots,x_r]$ denotes the ring of polynomials in variables $x_1,\dots,x_r$ with coefficients in $\bk$.

Given a $\bZ$-graded algebra $A = \bigoplus_e A_e$ and $d \in \bZ$, we write $A(d)$ for the graded shift with $A(d)_e = A_{d+e}$.

\section{The Veronese category}

\subsection{Polynomial rings}
Let $\bk$ be a commutative ring. Fix $r > 0$. 

To motivate the constructions in this section, we recall some definitions from \cite[\S 2.1]{secver}. There, we defined a bigraded space $\cA$ with $\cA_{d,n} = (\Sym^d \bk^r)^{\otimes n}$ which is equipped with two products. For the first, let $\sigma$ be a partition $\{i_1 < \cdots < i_n\} \cup \{j_1 < \cdots < j_m\}$ of $[n+m]$ into subsets of sizes $n$ and $m$. For $u_i, v_j \in \Sym^d \bk^r$, define
\begin{align*}
\cdot_\sigma \colon (\Sym^d \bk^r)^{\otimes n} \otimes (\Sym^d \bk^r)^{\otimes m} \to (\Sym^d \bk^r)^{\otimes (n+m)}
\end{align*}
by $(u_1 \otimes \cdots \otimes u_n) \cdot_\sigma (v_1 \otimes \cdots \otimes v_m) = w_1 \otimes \cdots \otimes w_{n+m}$ with $w_{i_k}=u_k$ and $w_{j_k}=v_k$.

The second multiplication is given by
\begin{align*}
* \colon (\Sym^d \bk^r)^{\otimes n} \otimes (\Sym^e \bk^r)^{\otimes n} &\to (\Sym^{d+e} \bk^r)^{\otimes n}\\
(u_1 \otimes \cdots \otimes u_n) \otimes (v_1 \otimes \cdots \otimes v_n) &\mapsto (u_1v_1) \otimes \cdots \otimes (u_n v_n).
\end{align*}
Implicit in this definition is that $\alpha* \beta = 0$ for $\alpha \in \cA_{d,n}$ and $\beta \in \cA_{e,m}$ when $n \ne m$. In \cite{secver}, we were interested in $\cA$ as a left-module over itself via these two multiplications: in our application to secant varieties, our primary interest was understanding what an element in some bigraded space $\cA_{d,m}$ generates under left multiplication by all other elements in all possible ways. The space $\cA_{d,m}$ can be thought of as the space of all possible operations between $\cA_{0,0} = \bk$ and $\cA_{d,m}$, and so $\cA$ is ``freely generated'' in bidegree $(0,0)$. For our applications to syzygies, we will need the spaces which are freely generated in other bidegrees.

Following \cite{fi-modules} and \cite{catgb}, we will encode the operations from $\cA_{d,m}$ to $\cA_{e,n}$ as the space of morphisms from an object $(d,m)$ to an object $(e,n)$ in an abstract category $\bk\cV_r$. The above two operations tell us how to do this when $d=e$ or when $m=n$. To be precise, when $d=e$, an operation from $\cA_{d,m}$ to $\cA_{d,n}$ is given by the partition $\sigma$ of $[n]$ and an element of $\cA_{d,n-m}$. A basis of these operations can be  encoded by an order-preserving injection $[m] \to [n]$ together with an ordered list of monomials, and we choose this latter perspective for consistency with the category ${\bf OI}$ studied in \cite[\S 7]{catgb}. When $m=n$, an operation from $\cA_{d,n}$ to $\cA_{e,n}$ consists of a choice of an element of $\cA_{e-d,n}$, and has a basis given by the monomials, which again are represented by an ordered list of $n-m$ monomials in $\Sym^d \bk^r$. Our preference is to represent a monomial in $\Sym^d \bk^r$ by its exponent vector in $\bZ_{\ge 0}^r$. 

When $d \ne e$ and $m \ne n$, succinctly describing a basis for the space of operations is more complicated; roughly, it should be a formal composition of the basic morphisms just mentioned modulo certain natural identifications. Our task now is to carefully give these definitions along with their compositions.

Define
\[
\multi{r}{d} = \{(x_1, \dots, x_r) \in \bZ_{\ge 0}^r \mid x_1 + \cdots + x_r = d\}.
\]
As mentioned, we think of this as a basis for $\Sym^d \bk^r$, and we define addition componentwise 
\[
+ \colon \multi{r}{d} \times \multi{r}{e} \to \multi{r}{d+e}.
\]

\begin{definition} \label{defn:veronese-cat}
Define the \df{Veronese category} $\cV_r$ as follows. The objects of $\cV_r$ are pairs $(d,m) \in \bZ_{\ge 0}^2$ and a morphism $\alpha \colon (d,m) \to (e,n)$ consists of the following data:
\begin{itemize}
\item An order-preserving injection $\alpha_1 \colon [m] \to [n]$,
\item A function $\alpha_2 \colon [n] \setminus \alpha_1([m]) \to \multi{r}{e}$,
\item A function $\alpha_3 \colon [m] \to \multi{r}{e-d}$.
\end{itemize}
In particular, $\hom_{\cV_r}((d,m), (e,n)) = \emptyset$ if $d > e$. Given another morphism $\beta \colon (e,n) \to (f,p)$, the composition $\beta \circ \alpha = \gamma \colon (d,m) \to (f,p)$ is defined by 
\begin{itemize}
\item $\gamma_1 = \beta_1 \circ \alpha_1$,
\item $\gamma_2 \colon [p] \setminus \gamma_1([m]) \to \multi{r}{f}$ is defined by:
  \begin{itemize}
  \item if $i \in [p] \setminus \beta_1([n])$, then $\gamma_2(i) = \beta_2(i)$, and
  \item if $i \in \beta_1([n] \setminus \alpha_1([m]))$, then $\gamma_2(i) = \alpha_2(i') + \beta_3(i')$ where $i'$ is the unique preimage of $i$ under $\beta_1$.
  \end{itemize}
\item $\gamma_3 \colon [m] \to \multi{r}{f-d}$ is defined by $\gamma_3(i) = \alpha_3(i) + \beta_3(\alpha_1(i))$. \qedhere
\end{itemize}
\end{definition}

When $d=e$, the function $\alpha_3$ is superfluous and the pair $(\alpha_1, \alpha_2)$ encodes an operation as discussed above. Similarly, when $n=m$, the functions $\alpha_1$ and $\alpha_2$ are superfluous, and $\alpha_3$ also encodes an operation as discussed above.

\begin{lemma}
Composition as defined above is associative.
\end{lemma}

\begin{proof}
Suppose we are given three morphisms 
\[
(d,m) \xrightarrow{\alpha} (e,n) \xrightarrow{\beta} (f,p) \xrightarrow{\gamma} (g,q).
\]
We will verify that all 3 components of both ways of interpreting $\gamma\beta\alpha$ are the same.
\begin{itemize}
\item $(\gamma_1 \beta_1) \alpha_1 = \gamma_1 (\beta_1  \alpha_1)$ by associativity of function composition.
\item Consider $[q] \setminus (\gamma_1 \beta_1 \alpha_1 ([m])) \to \multi{r}{g}$. 
  \begin{itemize}
  \item If $i \in [q] \setminus \gamma_1([p])$, then $i \mapsto \gamma_2(i)$ under both compositions.
  \item If $i \in \gamma_1([p] \setminus \beta_1([n]))$, 
let $i'$ be the unique preimage of $i$ under $\gamma_1$. 

Under $\gamma(\beta\alpha)$, we have $i \mapsto (\beta\alpha)_2(i') + \gamma_3(i') = \beta_2(i') + \gamma_3(i')$. 

Under $(\gamma \beta)\alpha$, we have $i \mapsto (\gamma\beta)_2(i) = \beta_2(i') + \gamma_3(i')$.

  \item If $i \in \gamma_1\beta_1([n] \setminus \alpha_1([m]))$, 
let $i'$ be the unique preimage of $i$ under $\gamma_1$, and let $i''$ be the unique preimage of $i'$ under $\beta_1$.

Under $\gamma(\beta\alpha)$, we have $i \mapsto (\beta\alpha)_2(i') + \gamma_3(i') = \alpha_2(i'') + \beta_3(i'') + \gamma_3(i')$.

Under $(\gamma\beta)\alpha$, we have $i \mapsto \alpha_2(i'') + (\gamma\beta)_3(i'') = \alpha_2(i'') + \beta_3(i'') + \gamma_3(i')$.
  \end{itemize}

\item Now consider the map $[m] \to \multi{r}{g-d}$. 

Under $\gamma(\beta \alpha)$, we have $i \mapsto (\beta\alpha)_3(i) + \gamma_3(\beta_1\alpha_1(i)) = \alpha_3(i) + \beta_3(\alpha_1(i)) + \gamma_3(\beta_1\alpha_1(i))$.

Under $(\gamma\beta) \alpha$, we have $i \mapsto \alpha_3(i) + (\gamma\beta)_3(\alpha_1(i)) = \alpha_3(i) + \beta_3(\alpha_1(i)) + \gamma_3(\beta_1\alpha_1(i))$. \qedhere
\end{itemize}
\end{proof}

\begin{remark} \label{rmk:OI}
For fixed $d$, the full subcategory of $\cV_r$ on the objects $(d,n)$ with $n$ varying is equivalent to the category ${\bf OI}_N$ from \cite[\S 7.1]{catgb} where $N = |\multi{r}{d}|$. So $\cV_r$ combines these all at once with $d$ varying, but there are additional morphisms between these subcategories that allow them to interact.
\end{remark}

Let $\bk\cV_r$ be the $\bk$-linearization of $\cV_r$, i.e., $\hom_{\bk\cV_r}(x,y) = \bk[\hom_{\cV_r}(x,y)]$. A \df{$\bk\cV_r$-module} is a functor from $\cV_r$ to the category of $\bk$-modules. Equivalently, a $\bk\cV_r$-module is a $\bk$-linear functor from $\bk\cV_r$ to the category of $\bk$-modules. Morphisms of $\bk\cV_r$-modules are natural transformations, and $\bk\cV_r$-modules form an abelian category where submodules, kernels, cokernels, etc. are computed pointwise.

\begin{remark}
We can define the $\bk$-linearization $\bk\cV_r$ by modifying Definition~\ref{defn:veronese-cat} so that the functions $\alpha_2$ and $\alpha_3$ take values in the space of degree $e$ or degree $d-e$ homogeneous polynomials in $\bk[x_1,\dots,x_r]$, respectively, if we impose certain multilinear relations on them. While this seems complicated, this is necessary in Definition~\ref{def:VB} where $\bk[x_1,\dots,x_r]$ is replaced by a general ring which does not possess a monomial basis.
\end{remark}

Given $(d,m) \in \bZ_{\ge 0}^2$, define a $\bk\cV_r$-module $P_{d,m}$ by $P_{d,m}(e,n) = \bk[\hom_{\cV_r}((d,m), (e,n))]$. This is the principal projective $\bk\cV_r$-module generated in bidegree $(d,m)$, and they give a set of projective generators for the category of $\bk\cV_r$-modules, i.e., every $\bk\cV_r$-module is a quotient of a direct sum of principal projectives. Then $P_{d,m}(e,n)$ is the space of operations from $\cA_{d,m}$ to $\cA_{e,n}$ which we discussed at the beginning of the section, so $P_{d,m}$ is an $\cA$-module freely generated in bidegree $(d,m)$.

To emphasize the category, we will write $P_{d,m}^{\cV_r}$. A $\bk\cV_r$-module $M$ is \df{finitely generated} if there is a surjection $\bigoplus_{i=1}^g P_{d_i, m_i} \to M \to 0$ with $g$ finite. A $\bk\cV_r$-module is \df{noetherian} if all of its submodules are finitely generated.

For the definition of a Gr\"obner category, see \cite[Definition 4.3.1]{catgb}. We will not need this concept except for its consequence in the next result.

\begin{proposition} \label{prop:ver-noeth}
$\cV_r$ is a Gr\"obner category. In particular, if $\bk$ is noetherian, then every finitely generated $\bk\cV_r$-module is noetherian.
\end{proposition}

\begin{subeqns}
\begin{proof}
We will use \cite[Theorem 4.3.2]{catgb}. So fix $(d,m) \in \bZ_{\ge 0}^2$. Let 
\[
\Sigma = \bZ_{\ge 0}^r \amalg \bZ_{\ge 0}^r = \coprod_{e \ge 0} \multi{r}{e} \amalg \coprod_{e \ge d} \multi{r}{e-d}.
\]
We order elements of each copy of $\bZ_{\ge 0}^r$ by componentwise comparison, i.e., $(a_1, \dots, a_r) \le (b_1, \dots, b_r)$ if and only if $a_i \le b_i$ for all $i$. Elements in the different copies of $\bZ_{\ge 0}^r$ are incomparable. 

Recall that a poset $P$ is noetherian if, given any infinite sequence $x_1, x_2, \dots$ with $x_i \in P$, there exists $i<j$ such that $x_i \le x_j$. Then $\bZ_{\ge 0}^r$ is noetherian by Dickson's lemma (alternatively, finite products of noetherian posets are noetherian, see \cite[Proposition 2.3]{catgb}). Clearly finite disjoint unions preserve noetherianity, so $\Sigma$ is also noetherian. 

Given a poset $P$, let $P^\star$ be the set of finite length sequences of elements in $P$. The Higman order on $P^\star$ is defined by $(x_1, \dots, x_r) \le (y_1, \dots, y_s)$ if there exists $1 \le i_1 < \cdots < i_r \le s$ such that $x_j \le y_{i_j}$ for $1 \le j \le r$. If $P$ is noetherian, then so is $P^\star$ by Higman's lemma (see \cite[Theorem 2.5]{catgb}). In particular, $\Sigma^\star$ is noetherian.

Given a morphism $\alpha \colon (d,m) \to (e,n)$, define a word $w(\alpha) \in \Sigma^\star$ of length $n$ by
\begin{align} \label{eqn:word}
w(\alpha)_i = \begin{cases} 
\alpha_3(j) \in \multi{r}{e-d} & \text{if $i = \alpha_1(j)$}\\
\alpha_2(i) \in \multi{r}{e} & \text{if $i \notin \alpha_1([m])$} 
\end{cases}.
\end{align}
Note that $\alpha$ can be reconstructed from $w(\alpha)$. Define $\alpha \le \gamma$ if there exists $\beta$ such that $\gamma = \beta \circ \alpha$. Then it follows from our definition of composition that the set of morphisms $\alpha \colon (d,m) \to (e,n)$ with $(d,m)$ fixed and $(e,n)$ varying is naturally a subposet of $\Sigma^\star$, i.e., $\alpha \le \alpha'$ if and only if $w(\alpha) \le w(\alpha')$. Since noetherianity is inherited by subposets, we conclude that this partial order on morphisms with source $(d,m)$ is noetherian. This is one of the conditions to check that $\cV_r$ is Gr\"obner.

To finish, we need to show that the set of morphisms with source $(d,m)$ is orderable, i.e., for each $(e,n)$, there exists a total ordering on the set of morphisms $(d,m) \to (e,n)$ so that for any $\beta \colon (e,n) \to (f,p)$, we have $\alpha < \alpha'$ implies that $\beta \alpha < \beta \alpha'$. To do this, first put the lexicographic order on $\bZ_{\ge 0}^r$ and then totally order $\Sigma$ by declaring all of the elements of the first $\bZ_{\ge 0}^r$ to be larger than all of the elements of the second $\bZ_{\ge 0}^r$. Now extend this to a lexicographic ordering on $\Sigma^\star$ to get the desired ordering.
\end{proof} 
\end{subeqns}

By our comments in the beginning of the section, \cite[Corollary 2.9]{secver} is a special case of Proposition~\ref{prop:ver-noeth}.

\begin{remark}
In the definition of morphisms in $\cV_r$, we could drop the requirement that $\alpha_1$ is order-preserving. The forgetful functor from $\cV_r$ to this category satisfies property (F) in the sense of \cite[Definition 3.2.1]{catgb}. Hence this new category is quasi-Gr\"obner and its finitely generated modules are noetherian when $\bk$ is noetherian. Furthermore, going back to Remark~\ref{rmk:OI}, the full subcategory of this category on objects $(d,n)$, with $d$ fixed and $n$ varying, is equivalent to the category ${\bf FI}_N$ from \cite[\S 7.1]{catgb} with $N = |\multi{r}{d}|$.
\end{remark}

\subsection{Symmetrized versions}

In $\bk\cV_r$, the space of morphisms $(0,0) \to (d,m)$ is identified with the tensor power $(\Sym^d \bk^r)^{\otimes m}$. For our applications, we need symmetric powers $\Sym^m(\Sym^d \bk^r)$, so we now define symmetrized versions of the Veronese category $\bk\cV_r$.

\begin{definition}
Given $\alpha \colon (d,m) \to (e,n)$ and $\sigma \in \Sigma_n$, there is a unique $\tau \in \Sigma_m$ so that $\sigma \alpha_1 \tau^{-1}$ is order-preserving; we refer to $\tau$ as the permutation induced by $\sigma$ with respect to $\alpha_1$. Define $\sigma(\alpha)$ by
\begin{itemize}
\item $\sigma(\alpha)_1 = \sigma \alpha_1 \tau^{-1}$,
\item $\sigma(\alpha)_2 = \alpha_2 \sigma^{-1}$,
\item $\sigma(\alpha)_3 = \alpha_3 \tau^{-1}$.
\end{itemize}
This defines an action of $\Sigma_n$ on $\hom_{\cV_r}((d,m), (e,n))$, and we set
\[
\hom_{\bk\cV_r^\Sigma}((d,m), (e,n)) = \bk[\hom_{\cV_r}((d,m), (e,n))]^{\Sigma_n}
\]
where the superscript denotes taking invariants. 
\end{definition}

\begin{lemma} \label{lem:sym-action}
Given $\alpha \colon (d,m) \to (e,n)$ and $\beta \colon (e,n) \to (f,p)$, and $\sigma \in \Sigma_p$, we have $\sigma(\beta \circ \alpha) = \sigma(\beta) \circ \tau(\alpha)$ where $\tau \in \Sigma_n$ is the permutation induced by $\sigma$ with respect to $\beta_1$. In particular, $\bk\cV_r^\Sigma$ is a $\bk$-linear subcategory of $\bk\cV_r$.
\end{lemma}

\begin{proof}
Let $\rho \in \Sigma_\ell$ be the permutation induced by $\tau$ with respect to $\alpha_1$. Then $(\sigma \beta_1 \tau^{-1}) (\tau \alpha_1 \rho^{-1})$ is order-preserving, so $\rho$ is also the permutation induced by $\sigma$ with respect to $\beta_1 \alpha_1$. Hence $\sigma(\beta \alpha)_1 = \sigma(\beta)_1 \tau(\alpha)_1$.

Next, we show that $\sigma(\beta \alpha)_2 = (\sigma(\beta) \tau(\alpha))_2$. If $i \in [p] \setminus \sigma (\beta)_1([n])$, then 
\begin{align*}
\sigma(\beta \alpha)_2(i) = (\beta \alpha)_2 \sigma^{-1}(i) = \beta_2(\sigma^{-1}(i)) = \sigma(\beta)_2(i) = (\sigma(\beta) \tau(\alpha))_2(i).
\end{align*}
Otherwise, if $i \in \sigma(\beta)_1([n] \setminus \tau(\alpha)_1([m]))$, let $i'$ be the unique preimage of $i$ under $\sigma \beta_1 \tau^{-1}$. Then $\tau^{-1}(i')$ is the unique preimage of $\sigma^{-1}(i)$ under $\beta_1$, and we have
\begin{align*}
\sigma(\beta \alpha)_2(i) = (\beta \alpha)_2 \sigma^{-1}(i) = \alpha_2(\tau^{-1}(i')) + \beta_3(\tau^{-1}(i')) = \tau(\alpha)_2(i') + \sigma(\beta)_3 (i') = (\sigma(\beta) \tau(\alpha))_2(i).
\end{align*}

Finally, we show that $\sigma(\beta \alpha)_3 = (\sigma(\beta) \tau(\alpha))_3$. For $i \in [\ell]$, we have 
\begin{align*}
(\sigma(\beta) \circ \tau(\alpha))_3(i) 
&= \tau(\alpha)_3(i) + \sigma(\beta)_3 (\tau(\alpha)_1(i))\\ 
&= \alpha_3(\rho^{-1}(i)) + \sigma(\beta)_3 (\tau \alpha_1 \rho^{-1}(i))\\
&= \alpha_3(\rho^{-1}(i)) + \beta_3 \alpha_1 (\rho^{-1}(i))\\
&= (\beta \alpha)_3 (\rho^{-1}(i)) =\sigma(\beta \alpha)_3(i).  \qedhere
\end{align*}
\end{proof}

A $\bk\cV_r^\Sigma$-module is a $\bk$-linear functor from $\bk\cV_r^\Sigma$ to the category of $\bk$-modules. For each $(d,m)$, the principal projective $\bk\cV_r^\Sigma$-module is defined by $P_{d,m}^{\bk\cV_r^\Sigma}(e,n) = \hom_{\bk \cV_r^\Sigma}((d,m), (e,n))$ and we say that a $\bk\cV_r^\Sigma$-module $M$ is finitely generated if there is a surjection $\bigoplus_{i=1}^g P_{d_i, m_i}^{\bk\cV_r^\Sigma} \to M \to 0$ with $g$ finite.

\begin{proposition} 
If $\bk$ contains a field of characteristic $0$, then every finitely generated $\bk\cV_r^\Sigma$-module is noetherian.
\end{proposition}

\begin{proof}
Set $P_{d,m} = P_{d,m}^{\cV_r}$ and $Q_{d,m} = P_{d,m}^{\bk\cV_r^\Sigma}$; we have a natural inclusion $Q_{d,m}(e,n) \subseteq P_{d,m}(e,n)$ for all $(e,n)$. Given a $\bk\cV_r^\Sigma$-submodule $M$ of $Q_{d,m}$, let $N$ be the $\cV_r$-submodule of $P_{d,m}$ that it generates. Given a list of generators of $N$ coming from $M$, Proposition~\ref{prop:ver-noeth} guarantees that some finite subset $\gamma_1, \dots, \gamma_g$ of them already generates $N$. Let $\pi$ be the symmetrization map 
\begin{align*}
\bk[\hom_{\cV_r}((d,m), (e,n))] &\to \bk[\hom_{\cV_r}((d,m), (e,n))]^{\Sigma_n},\\
\alpha &\mapsto \frac{1}{n!} \sum_{\sigma \in \Sigma_n} \sigma(\alpha).
\end{align*}
If $\alpha \in \bk[\hom_{\cV_r}((d,m), (e,n))]^{\Sigma_n}$, then $\pi(\alpha) = \alpha$; given $\beta \in \bk[\hom_{\cV_r}((e,n), (f,p))]$, then $\pi(\beta \alpha) = \pi(\beta) \alpha$ by Lemma~\ref{lem:sym-action}. 

Given any element $\gamma$ of $M$, we have an expression $\gamma = \sum_i \delta_i \gamma_i$ where $\delta_i \in \bk\cV_r$. So, applying $\pi$, we get $\gamma = \pi(\gamma) = \sum_i \pi(\delta_i) \gamma_i$, so $\gamma_1,\dots,\gamma_g$ also generate $M$ as a $\bk\cV_r^\Sigma$-module. In particular, the principal projectives of $\bk\cV_r^\Sigma$ are noetherian, so the same is true for any finitely generated module.
\end{proof}

Now assume that $\bk$ contains a field of characteristic $0$. We define the \df{symmetrized Veronese category} $\sV_r = (\bk\cV_r)_\Sigma$ as follows. First, set
\[
\hom_{\sV_r}((d,m), (e,n)) = \bk[\hom_{\cV_r}((d,m), (e,n))]_{\Sigma_n}
\]
where the subscript denotes coinvariants under $\Sigma_n$. We have an isomorphism
\begin{align*}
\bk[\hom_{\cV_r}((d,m), (e,n))]_{\Sigma_n} &\xrightarrow{\cong} \bk[\hom_{\cV_r}((d,m), (e,n))]^{\Sigma_n}\\
\alpha &\mapsto \frac{1}{n!} \sum_{\sigma \in \Sigma_n} \sigma(\alpha),
\end{align*}
and we use this to transfer the $\bk$-linear category structure from $\bk\cV_r^\Sigma$ to $\sV_r$. Note, in particular, that $\hom_{\sV_r}((0,0), (d,m))$ is identified with $\Sym^m(\Sym^d \bk^r)$, which was our goal mentioned at the beginning of this section.

\subsection{General rings}
Let $B$ be a graded $\bk$-algebra which is generated by $B_1$ and assume that $B_1$ is a quotient of $\bk^r$ with $r<\infty$. In \cite[\S\S 2.2, 3]{secver}, we defined a generalization $\cB$ of the algebra $\cA$ which satisfies $\cB_{d,n} = B_d^{\otimes n}$ and also has two multiplications $\cdot_\sigma$ and $*$. We also defined the symmetric analogue $\cB_\Sigma$ which satisfies $(\cB_\Sigma)_{d,n} = \Sym^n(B_d)$. This was applied to secant varieties of Veronese embeddings of general varieties (other than projective spaces). As in the previous sections, we want to translate the definition of the algebras $\cB$ and $\cB_\Sigma$ into categorical language so that we can discuss modules and apply it to the study of syzygies of secant varieties. The main technical difficulty is that $B$ in general does not have a monomial basis, so rather than define a category and linearize it (as was done for $\cV_r$), we have to start immediately with a $\bk$-linear category.

\begin{definition} \label{def:VB}
Define a $\bk$-linear category $\bk\cV_B$ as follows. Objects of $\bk\cV_B$ are elements of $\bZ_{\ge 0}^2$, and a basic morphism $\alpha \colon (d,m) \to (e,n)$ consists of the following data:
\begin{itemize}
\item An order-preserving injection $\alpha_1 \colon [m] \to [n]$,
\item A function $\alpha_2 \colon [n] \setminus \alpha_1([m]) \to B_e$,
\item A function $\alpha_3 \colon [m] \to B_{e-d}$.
\end{itemize}
Then $\hom_{\bk\cV_B}((d,m), (e,n))$ is the free $\bk$-module generated by basic morphisms modulo the following ``multilinear'' relations:
\begin{itemize}
\item Let $\beta \colon (d,m) \to (e,n)$ be a morphism with $\beta_1 = \alpha_1$, $\beta_2 = \alpha_2$ and $\beta_3(i) = \alpha_3(i)$ for all $i \in [m] \setminus \{j\}$. Then $\alpha + \beta = \gamma$ where $\gamma_1 = \beta_1$, $\gamma_2 = \beta_2$, and $\gamma_3(i) = \beta_3(i)$ for $i \in [m] \setminus \{j\}$ and $\gamma_3(j) = \alpha_3(j) + \beta_3(j)$. 

Furthermore, if $\beta_3(j) = \xi \alpha_3(j)$ for $\xi \in \bk$, then $\xi \alpha = \beta$.

\item Let $\beta \colon (d,m) \to (e,n)$ be a morphism with $\beta_1 = \alpha_1$, $\beta_2 = \alpha_2$ and $\beta_3(i) = \alpha_3(i)$ for all $i \in ([n] \setminus \alpha_1([m])) \setminus \{j\}$. Then $\alpha + \beta = \gamma$ where $\gamma_1 = \beta_1$, $\gamma_3 = \beta_3$, and $\gamma_2(i) = \beta_2(i)$ for $i \in ([n] \setminus \alpha_1([m])) \setminus \{j\}$ and $\gamma_2(j) = \alpha_2(j) + \beta_2(j)$. 

Furthermore, if $\beta_2(j) = \xi \alpha_2(j)$ for $\xi \in \bk$, then $\xi \alpha = \beta$.
\end{itemize}

In particular, 
\[
\hom_{\bk\cV_B}((d,m), (e,n)) \cong (B_{e-d}^{\otimes m} \otimes_\bk B_e^{\otimes (n-m)})^{\oplus \binom{n}{m}}.
\]
More precisely, the morphisms with a given $\alpha_1$ belong to the tensor product of $m$ copies of $B_{e-d}$ and $n-m$ copies of $B_e$ where the $i$th space is $B_{e-d}$ if and only if $i$ in the image of $\alpha_1$.

Composition of morphisms is defined as in $\cV_r$: the main change is to replace $+$ with multiplication in $B$ (in $\cV_r$, we were dealing with the monomial basis of $B = \bk[x_1,\dots,x_r]$ so we used their exponent vectors).
\end{definition}

Similarly, we can define $\bk\cV_B^\Sigma$ and $\bk\sV_B$. So we have an identification
\[
\hom_{\bk\sV_B}((d,m), (e,n)) = \Sym^m(B_{e-d}) \otimes_\bk \Sym^{n-m}(B_e),
\]
and the principal projectives in $\bk\cV_B$ and $\bk\sV_B$ are naturally quotients of the corresponding principal projectives in $\cV_r$ and $\sV_r$. So we get the following result:

\begin{proposition} \label{prop:symver-noeth}
Suppose $\bk$ is a field of characteristic $0$. Every finitely generated $\bk\cV_B$-module is noetherian, and the same is true for finitely generated $\bk\sV_B$-modules.
\end{proposition}

\begin{remark}
These definitions parallel the constructions in \cite[\S 3]{secver}. In particular, we can identify $\cB^\Sigma$ and $\cB_\Sigma$ from \cite{secver} with the principal projectives generated in degree $(0,0)$ in $\bk\cV_B^\Sigma$ and $\bk \sV_B$, respectively. Furthermore, the notions of ideal and di-ideal in \cite{secver} translate to submodules in both cases. So \cite[Proposition 3.3]{secver} is a special case of Proposition~\ref{prop:symver-noeth}.
\end{remark}

\section{Syzygies of secant ideals}

In this section, $\bk$ is a field of characteristic $0$ and $B$ is a graded $\bk$-algebra generated by $B_1$ with $\dim_\bk B_1 < \infty$.

The principal projective $P_{0,0}$ in $\bk\sV_B$ is the algebra $\cB_\Sigma$ from \cite{secver} and each principal projective $P_{d,m}$ is a module over it. We use $\cB_\Sigma(-d,-m)$ to denote this module; by Proposition~\ref{prop:symver-noeth}, these are all noetherian modules. 

In \cite[\S 4]{secver}, we defined $\cI_B(1) \subset \cB_\Sigma$ by setting $\cI_B(1)_{d,m}$ to be the kernel of the multiplication map $\Sym^m(B_d) \to B_{dm}$, and we inductively defined the \df{secant ideals} $\cI_B(r)$ by setting $\cI_B(r)_{d,m}$ to be the kernel of the comultiplication map 
\[
\Sym^m(B_d) \to \bigoplus_{i=0}^m \Sym^i(B_d) / \cI_B(1)_{d,i} \otimes_\bk \Sym^{m-i}(B_d) / \cI_B(r-1)_{d,m-i}.
\]
By \cite[Proposition 4.3]{secver}, $\cI_B(r)$ is a $\bk\sV_B$-submodule of $P_{0,0}$ for all $r$.

\begin{remark}
Actually, \cite[Proposition 4.3]{secver} shows that assigning the kernel of $B_d^{\otimes m} \to B_{dm}$ to $(d,m)$ is a functor on $\bk \cV_B$, in which case one can work over any $\bk$ (not just fields of characteristic $0$). This is the preimage of $\cI_B(1)$ under the quotient map $\cB \to \cB_\Sigma$.
Unfortunately, this map is not compatible with all of the relevant algebraic operations, so the preimage of $\cI_B(2)$ need not be a submodule (see Example~\ref{eg:rnc}).
\end{remark}

For $d$ fixed, $\bigoplus_m \cI_B(r)_{d,m}$ is an ideal in $\Sym(B_d)$. So we can define an algebra
\[
\Sec_{d,r}(B) = \bigoplus_{m \ge 0} \Sym^m(B_d) / \cI_B(r)_{d,m}
\]
which is a quotient of $\Sym(B_d)$. More generally, if $M$ is a $\bk \sV_B$-module, then for $d$ fixed, $\bigoplus_m M_{d,m}$ is a $\Sym(B_d)$-module.

\begin{lemma} \label{lem:free-module}
Fix $d,e,n$. Then 
\[
\bigoplus_{m \ge 0} \cB_\Sigma(-e, -n)_{d,m}
\]
is a free $\Sym(B_d)$-module generated in degree $n$ whose rank is $\dim_\bk \Sym^n(B_{d-e})$.
\end{lemma}

\begin{proof}
We have
\begin{align*}
\bigoplus_{m \ge 0} \cB_\Sigma(-e, -n)_{d,m} &= \bigoplus_{m \ge 0} (\Sym^{n}(B_{d - e}) \otimes \Sym^{m - n}(B_d)) \\
&= \Sym^{n}(B_{d-e}) \otimes \Sym(B_d)(-n).
\end{align*}
As follows from the definitions, the action of $\Sym(B_d)$ on this space corresponds to the usual multiplication on $\Sym(B_d)(-n)$.
\end{proof}

\begin{theorem}
There is a function $C_B(i,r)$, depending on $i,r,B$, but independent of $d$, such that $\Tor_i^{\Sym(B_d)}(\Sec_{d,r}(B), \bk)$ is concentrated in degrees $\le C_B(i,r)$.
\end{theorem}

\begin{proof}
$\cI_B(r)$ is a finitely generated submodule of $\cB_\Sigma$, and hence has a projective resolution $\cdots\to \bF_i \to \bF_{i-1} \to \cdots \to \bF_0$ such that each $\bF_i$ is a finite direct sum of principal projective modules by Proposition~\ref{prop:symver-noeth}. For $d$ fixed, we get an exact complex of $\Sym(B_d)$-modules
\[
\cdots \to \bigoplus_m (\bF_i)_{d,m} \to \bigoplus_m (\bF_{i-1})_{d,m} \to \cdots \to \bigoplus_m (\bF_0)_{d,m} \to \Sec_{d,r}(B) \to 0.
\]
If $\bF_i = \bigoplus_{j=1}^k \cB_\Sigma(-d_j,-m_j)$, then set $C_B(i,r) = \max(m_1,\dots,m_k)$. In particular, by Lemma~\ref{lem:free-module}, this gives a free resolution which can be used to compute $\Tor_i^{\Sym(B_d)}(\Sec_{d,r}, \bk)$ which we conclude is concentrated in degrees $\le C_B(i,r)$.
\end{proof}

\begin{remark} \label{rmk:delta-mod}
Write $T_{i;d,r}(B) = \Tor_i^{\Sym(B_d)}(\Sec_{d,r}(B), \bk)$. As used above, this is $\bZ$-graded, and we denote the $m$th graded component by $T_{i;d,r}(B)_m$. For fixed $i,m,r$, we get a functor on the full subcategory of $\bk \sV_B$ on objects of the form $(d,m)$ by $(d,m) \mapsto T_{i;d,r}(B)_m$. From the results above, we conclude that this is a finitely generated functor. In particular, as we allow $d$ to vary, this says that $T_{i;d,r}(B)_m$ is ``built out'' of $T_{i;d',r}(B)_m$ where the $d'$ range over some finite list of integers. This can be thought of as the Veronese analogue of $\Delta$-modules in the sense of Snowden (\cite{delta-mod}, \cite[\S 9]{catgb}).
\end{remark}

If $V$ is a vector space and $X \subseteq V$ is a conical subscheme, i.e., an affine scheme whose defining ideal is homogeneous, then we can take $B = \Sym(V) / I_X$ where $I_X$ is the ideal of $X$ (see \cite[\S 5]{secver} for some further discussion of these definitions). We write $\Sec_{d,r}(X)$ instead of $\Sec_{d,r}(\Sym(V)/I_X)$. The above result then specializes to the following result.

\begin{corollary} \label{cor:sec-syz}
Let $X \subseteq V$ be a conical subscheme. There is a function $C_X(i,r)$, depending on $i,r,X$, but independent of $d$, such that the $i$th Tor module of $\Sec_{d,r}(X)$ is concentrated in degrees $\le C_X(i,r)$.
\end{corollary}

The case $i=1$ refers to the minimal equations generating the ideal of $\Sec_{d,r}(X)$, so Corollary~\ref{cor:sec-syz} generalizes \cite[Theorem 5.1]{secver}.

For the next proof, we use Schur functors $\bS_\lambda$ which are indexed by integer partitions (see \cite[\S 5]{expos} and the references given there). In particular, $\bS_d$ is the $d$th symmetric power. We use $\ell(\lambda)$ to denote the number of nonzero parts of $\lambda$ (for symmetric powers, $\ell=1$ since $(d)$ just has one entry). If $\bS_\nu \subseteq \bS_\lambda \otimes \bS_\mu$, then $\ell(\nu) \le \ell(\lambda) + \ell(\mu)$ \cite[\S 9.1]{expos}.

\begin{theorem} \label{thm:projspace}
In the case $X = V$, the function $C_X(i,r)$ is independent of $\dim V$ once $\dim V \ge r+i$. In particular, there is a bound that works for all $V$ simultaneously.
\end{theorem}

\begin{proof}
$\Sec_{d,r}(V)$ is a direct sum of Schur functors $\bS_\lambda(V^*)$ with $\ell(\lambda) \le r$ since it is contained in a subspace variety (see \cite[Remark 5.2]{secver}). The $i$th Tor module of $\Sec_{d,r}(V)$ is the $i$th homology of the Koszul complex on $\Sym^d(V^*)$ tensored (over $\Sym(\Sym^d(V^*))$) with $\Sec_{d,r}(V)$, so is a subquotient of $\bigwedge^i(\Sym^d(V^*)) \subset (\Sym^d(V^*))^{\otimes i}$ tensored (over $\bk$) with $\Sec_{d,r}(V)$. So all Schur functors $\bS_\mu(V^*)$ that appear in the $i$th Tor module satisfy $\ell(\mu) \le r+i$ by the subadditivity of $\ell$ mentioned above. In particular, no information is lost by specializing to the case $\dim V = r+i$ \cite[Corollary 9.1.3]{expos}.
\end{proof}

\begin{remark}
The ideal of the Veronese embedding of any projective space has a quadratic Gr\"obner basis (for example, see \cite[Proposition 17]{ERT}). By semicontinuity of Tor with respect to taking initial ideals \cite[Theorem 8.29]{CCA} and the basic properties of the Taylor resolution for monomial ideals \cite[\S 6.1]{CCA}, this implies that $C_V(i,1) \le 2i$ for all $i$.
\end{remark}

\begin{example} \label{eg:rnc}
Consider the case of $B = \bk[s,t]$. Let $x_i = s^i t^{4-i}$ for $i=0,\dots,4$ and $y_j = s^j t^{5-j}$ for $j=0,\dots,5$, so that $x_0, \dots, x_4$ and $y_0, \dots, y_5$ are the coordinates of the $4$th and $5$th Veronese embeddings of $\bP^1$. In each case, the ideal of the second secant variety is generated by the $3 \times 3$ minors of the catalecticant matrices 
\[
X = \begin{pmatrix} x_0 & x_1 & x_2 \\ x_1 & x_2 & x_3 \\ x_2 & x_3 & x_4 \end{pmatrix}, \qquad
Y = \begin{pmatrix} y_0 & y_1 & y_2 & y_3 \\ y_1 & y_2 & y_3 & y_4 \\ y_2 & y_3 & y_4 & y_5 \end{pmatrix} 
\]
(see \cite[Exercise 6.2]{eisenbud-syzygies}). Consider the following preimage of $\det X$ in $B_4^{\otimes 3}$ (we omit the $\otimes$ symbol):
\[
x_0 x_2 x_4 - x_0 x_3 x_3 + x_1 x_2 x_3 - x_1 x_1 x_4 + x_2 x_1 x_3 - x_2 x_2 x_2.
\]
If we multiply this by $tst \in B_1^{\otimes 3}$ (in the categorical language, apply the morphism $\alpha \colon (4,3) \to (5,3)$ where $\alpha_3 \colon [3] \to B_1$ is the function $1 \mapsto t$, $2 \mapsto s$, $3 \mapsto t$), we get the following element in $B_5^{\otimes 3}$:
\[
y_0 y_3 y_4 - y_0 y_4 y_3 + y_1 y_3 y_3 - y_1 y_2 y_4 + y_2 y_2 y_3 - y_2 y_3 y_2.
\]
If we symmetrize this, the corresponding element in $\Sym^3(B_5)$ is not in the ideal generated by the $3 \times 3$ minors of $Y$. On the other hand, if we multiply it by $stt+tst+tts \in (B_1^{\otimes 3})^{\Sigma_3}$, the result is in the ideal generated by the $3 \times 3$ minors of $Y$ (which follows from \cite[Proposition 4.3]{secver}). Both of these statements can be checked with Macaulay2 \cite{M2}:
\begin{verbatim}
R=QQ[y0,y1,y2,y3,y4,y5];
I=minors(3, matrix{{y0,y1,y2,y3}, {y1,y2,y3,y4}, {y2,y3,y4,y5}});
f1=y0*y3*y4 - y0*y4*y3 + y1*y3*y3 - y1*y2*y4 + y2*y2*y3 - y2*y3*y2; 
f2=y1*y2*y4 - y1*y3*y3 + y2*y2*y3 - y2*y1*y4 + y3*y1*y3 - y3*y2*y2;
f3=y0*y2*y5 - y0*y3*y4 + y1*y2*y4 - y1*y1*y5 + y2*y1*y4 - y2*y2*y3;
I+ideal(f1) == I --false
I+ideal(f1+f2+f3) == I --true
\end{verbatim} 
\end{example}

\section{Complements}

As in \cite{secver}, we can also prove multi-graded versions of the above results. To be precise, let $B = \bigoplus_{\bd \in \bZ_{\ge 0}^\ell} B_\bd$ be a multi-graded ring which is generated by the $B_\bd$ with $|\bd|=1$. It is straightforward to generalize the definitions of the categories $\bk\cV_B$, $\bk\cV_B^\Sigma$, $\bk\sV_B$ by allowing the objects to be $(\bd, m) \in \bZ_{\ge 0}^\ell \times \bZ_{\ge 0}$ and proving analogous results as above. Since the notation becomes more messy, we will leave this to the interested reader.

The ensuing results can be stated as follows. Define a $\bk\sV_B$-module $\cI_B(1)$ by setting $\cI_B(1)_{\bd,m}$ to be the kernel of the multiplication map $\Sym^m(B_\bd) \to B_{m\bd}$ and inductively define the secant ideals $\cI_B(r)$ by setting $\cI_B(r)_{\bd,m}$ to be the kernel of the comultiplication map
\[
\Sym^m(B_\bd) \to \bigoplus_{i=0}^m \Sym^i(B_\bd) / \cI_B(1)_{\bd,i} \otimes_\bk \Sym^{m-i}(B_\bd) / \cI_B(r-1)_{\bd,m-i}.
\]
As usual, these are $\bk\sV_B$-submodules of the principal projective $P_{0,0}$, and we define 
\[
\Sec_{\bd, r}(B) = \bigoplus_{m \ge 0} \Sym^m(B_\bd) / \cI_B(r)_{\bd, m}.
\]

\begin{theorem}
There is a function $C_B(i,r)$, depending on $i,r,B$, but independent of $\bd$, such that $\Tor_i^{\Sym(B_\bd)}(\Sec_{\bd,r}(B), \bk)$ is concentrated in degrees $\le C_B(i,r)$.
\end{theorem}

A special case of this arises by considering a projective scheme $X$ and line bundles $\cL_1, \dots, \cL_\ell$ on $X$. We get a multi-graded ring $\bigoplus_{d_1, \dots, d_\ell \ge 0} \rH^0(X; \cL_1^{\otimes d_1} \otimes \cdots \otimes \cL_\ell^{\otimes d_\ell})$ and we can take $B$ to be the subring generated by $\rH^0(X; \cL_i)$ for $i=1,\dots,\ell$.

As an even more special case of this, we can consider partial flag varieties. Let $\be = (e_1 < e_2 < \cdots < e_\ell)$ be an increasing sequence of positive integers. Given a vector space $V$, let $\bF(\be, V)$ be the partial flag variety of type $\be$: its points are increasing sequences of subspaces $W_1 \subset \cdots \subset W_\ell \subset V$ where $\dim W_i = e_i$. Its Picard group is isomorphic to $\bZ^\ell$, generated by line bundles $\cL_1, \dots, \cL_\ell$ (corresponding to the fundamental weights $\omega_{e_1}, \dots, \omega_{e_\ell}$ of $\SL(V)$); set $\cL(\bd) := \cL_1^{\otimes d_1} \otimes \cdots \otimes \cL_\ell^{\otimes d_\ell}$. When $\be = (1)$, $\bF(\be,V)$ is projective space. Define $\Sec_{\bd,r}(\bF(\be,V))$ to be the $r$th secant variety of the image of $\bF(\be,V)$ under the map defined by the line bundle $\cL(\bd)$.

\begin{theorem} \label{thm:flag}
There is a function $C_\be(i,r)$, depending on $\be, i, r$, but independent of $\bd$, such that the $i$th Tor module of $\Sec_{\bd,r}(\bF(\be,V))$ is concentrated in degrees $\le C_\be(i,r)$. In fact, $C_\be(i,r)$ is also independent of $\dim V$ once $\dim V \ge r e_\ell + i$.
\end{theorem}

As before, \cite[Theorem 1.3]{secver} is the case $i=1$ of this theorem.

\begin{proof}
The only thing that has not already been explained is the independence of $C_\be(i,r)$ from $\dim V$. This is similar to the proof of Theorem~\ref{thm:projspace}, but we use instead the statement that the Schur functors $\bS_\lambda$ appearing in the coordinate ring satisfy $\ell(\lambda) \le r e_\ell$ \cite[Proposition 5.6]{manivel-michalek}.
\end{proof}

As in \cite{secver}, we can also take finite products of projective schemes. This is essentially covered by the case of multi-graded rings already. One slightly new case would be to take finite products of partial flag varieties and getting independence of the bounds from the dimensions of the vector spaces involved. These are the higher Tor group analogues of the results in \cite[\S 6]{secver}.

\end{document}